\RequirePackage[l2tabu, orthodox]{nag}

\documentclass[12pt]{amsart}
\usepackage{fullpage,url,amssymb,enumerate}
\usepackage[all]{xy} 
\usepackage{comment}
\usepackage{graphicx}

\usepackage[OT2,T1]{fontenc}

\usepackage{amsthm}

\usepackage{color}

\usepackage{tikz}

\def\bC{\mathbb{C}}

\def\cL{\mathcal{L}}

\def\bP{\mathbb{P}}
\def\cP{\mathcal{P}}

\def\bR{\mathbb{R}}

\def\bZ{\mathbb{Z}}

\def\vr{\overrightarrow{r}}
\def\vx{\overrightarrow{x}}
\def\vz{\overrightarrow{z}}

\usepackage{mathdots}

\DeclareMathOperator{\Fl}{Fl}
\DeclareMathOperator{\FM}{FM}
\DeclareMathOperator{\Gr}{Gr}
\DeclareMathOperator{\GZ}{GZ}

\DeclareMathOperator{\id}{id}

\DeclareMathOperator{\Perm}{Perm}

\DeclareMathOperator{\rk}{rank}

\DeclareMathOperator{\Spec}{Spec}

\newtheorem{thm}{Theorem}[subsection]
\newtheorem{lem}[thm]{Lemma}
\newtheorem{cor}[thm]{Corollary}
\newtheorem{prop}[thm]{Proposition}

\newtheorem{rem}[thm]{Remark}
\newtheorem{ex}[thm]{Example}
\newtheorem{defn}[thm]{Definition}

\makeatletter
\g@addto@macro\bfseries{\boldmath} 
\makeatother

\usepackage{microtype}

\usepackage[
	backref,
	pdfauthor={}, 
	pdftitle={hhmp_decomp},
]{hyperref}

\begin{document}

\title{The HHMP decomposition of the permutohedron and degenerations of torus orbits in flag varieties}

\author{Carl Lian}

\address{Department of Mathematics, Tufts University
\hfill \newline\texttt{}
 \indent Medford, MA, USA} \email{Carl.Lian@tufts.edu}

\begin{abstract}
Let $Z\subset\Fl(n)$ be the closure of a generic torus orbit in the full flag variety. Anderson-Tymoczko express the cohomology class of $Z$ as a sum of classes of Richardson varieties. Harada-Horiguchi-Masuda-Park give a decomposition of the permutohedron, the moment map image of $Z$, into subpolytopes corresponding to the summands of the Anderson-Tymoczko formula. We construct an explicit toric degeneration inside $\Fl(n)$ of $Z$ into Richardson varieties, whose moment map images coincide with the HHMP decomposition, thereby obtaining a new proof of the Anderson-Tymoczko formula.
\end{abstract}

\maketitle

\section{Introduction}

Let $\bC^n$ be a complex vector space of dimension $n$ with the standard action of an $n$-dimensional torus $T$. Let $\Fl(n)$ be the variety of complete flags in $\bC^n$, which inherits a standard $T$-action. Let $Z\subset\Fl(n)$ be the closure of the $T$-orbit of a generic point. The cycle class $[Z]$ in $H^{*}(\Fl(n))$ was computed by Anderson-Tymoczko. (We work throughout with rational coefficients.)

\begin{thm}\cite{at}\label{at_formula}
We have
\begin{equation*}
[Z]=\sum_{w\in S_{n-1}}\sigma_{\iota(w)}\sigma_{\overline{\iota}(w_0w)}
\end{equation*}
in $H^{(n-1)(n-2)}(\Fl(n))$.
\end{thm}

See \S\ref{sec:prelim} for notation. The class $[Z]$ is equal to that of a \emph{regular semisimple Hessenberg variety} in $\Fl(n)$ with Hessenberg function $h(i)=i+1$, which in turn is cut out by degeneracy loci whose classes are determined by the work of Fulton \cite{fulton}. The coefficients of $[Z]$ when expressed in the Schubert basis were previously determined by Klyachko \cite[Theorem 4]{k} in terms of representation theory, and a positive combinatorial interpretation was recently given by Nadeau-Tewari \cite{nt_forest}.

The form of Theorem \ref{at_formula} suggests that one might hope for the existence of a (toric) degeneration of $Z$ into the union of Richardson varieties (generically transverse intersections of two Schubert varieties) $Z_w$ in $\Fl(n)$ of class $\sigma_{\iota(w)}\sigma_{\overline{\iota}(w_0w)}$. The purpose of this paper is to construct such a degeneration, thereby giving a new proof of Theorem \ref{at_formula}.

We show moreover that this degeneration is already encoded in a polyhedral decompositition of the permutohedron $\Perm(n)$ given by Harada-Horiguchi-Masuda-Park \cite{hhmp}. Their decomposition, which we refer to as the HHMP decomposition, is given by a union
\begin{equation*}
\Perm(n)=\bigcup_{w\in S_{n-1}}\GZ(w)
\end{equation*}
of subpolytopes, indexed by permutations $w\in S_{n-1}$, which are faces of the \emph{Gelfand-Zetlin} polytope. The correspondence between the subpolytopes $\GZ(w)$ and the classes $\sigma_{\iota(w)}\sigma_{\overline{\iota}(w_0w)}$ appearing in the Anderson-Tymoczko formula is also shown to be volume-preserving, suitably interpreted.

More precisely, we identify the subpolytopes $\GZ(w)$ with the moment map images of the components $Z_w$ appearing in the special fiber of our degeneration. The $Z_w$ are themselves orbit closures of special points in $\Fl(n)$, and their moment map images are therefore \emph{flag matroid polytopes} $\FM(w)$. We summarize the new results below.

\begin{thm}
There exists an embedded toric degeneration of $Z\subset\Fl(n)$ into irreducible components $Z_w\subset\Fl(n)$, which are equal to $T$-orbit closures of special flags $\cL_w\in \Fl(n)$, and which all appear with multiplicity 1. Furthermore:
\begin{enumerate}
\item (Theorem \ref{FM=GZ}) $Z_w$ is the $T$-orbit closure of a special flag $\cL_w\in \Fl(n)$, whose associated 
flag matroid polytope $\FM(w)$ is equal to the polytope $\GZ(w)$ appearing in the HHMP decomposition.
\item (Theorem \ref{orbit=richardson}) $Z_w$ is a Richardson variety of class $\sigma_{\iota(w)}\sigma_{\overline{\iota}(w_0w)}$.
\end{enumerate}
\end{thm}

Combining these conclusions yields a new proof of Theorem \ref{at_formula}. Theorems \ref{FM=GZ} and \ref{orbit=richardson} are largely implicit in the work of other authors. Nadeau-Tewari \cite[\S 6-7]{nt_remixed} identify the polytopes $\GZ(w)$ with \emph{Bruhat interval polytopes} in the sense of Tsukerman-Williams \cite{tw}. Bruhat interval polytopes are in turn flag matroid polytopes \cite[Proposition 2.9]{tw} and moment map images of Richardson varieties \cite[Remark 7.11]{tw}. Our main contribution is the explicit construction of the degeneration witnessing the Anderson-Tymoczko formula; we give self-contained proofs of Theorems \ref{FM=GZ} and \ref{orbit=richardson} for completeness.

It may be of interest to construct similar degenerations for the more general Hessenberg varieties considered by Anderson-Tymoczko, whose classes are computed by a similar formula, or for torus orbit closures in other Lie types, where we are not aware of such formulas. Outside of the toric setting, we have carried out similar degenerations on the Grassmannians in connection with counting curves on projective spaces, see \cite[\S 8]{l_coll}.

We review preliminaries in \S\ref{sec:prelim} and the HHMP decomposition in \S\ref{sec:hhmp}. We construct the degeneration of $Z$ in \S\ref{sec:explicit_degen}. We prove Theorem \ref{FM=GZ} in \S\ref{sec:fmp}, which implies that the $Z_w$ are the only components that appear in our degeneration. We prove Theorem \ref{orbit=richardson} in \S\ref{sec:at}, which completes the proof of Theorem \ref{at_formula}. We explain how our degeneration pushes forward to Grassmannians in \S\ref{sec:grass}, recovering a result of Berget-Fink \cite[Theorem 5.1]{bf}.

\subsection{Acknowledgments}

Portions of this project were completed with the support of an NSF postdoctoral fellowship, DMS-2001976, an AMS-Simons travel grant, and during the author's visit to SLMath (formerly MSRI). We are grateful to Philippe Nadeau and Vasu Tewari for many helpful comments and references to the literature, to Matt Larson for pointing out an oversight in the original draft of this work, and to the referees for their suggestions for various improvements to the exposition.

\section{Preliminaries}\label{sec:prelim}

\subsection{Permutations}

Let $[n]=\{1,2,\ldots,n\}$. Permutations $w\in S_n$ are understood to be functions $w:[n]\to [n]$. 
\begin{defn}
Let $w\in S_{n}$ be a permutation. For $j\in[n]$, we define $r_j(w)\in[1,n+1-j]$ to be the integer such that $w^{-1}(j)$ is the $r_j(w)$-th largest integer among $w^{-1}(j),w^{-1}(j+1),\ldots,w^{-1}(n)$. We also write $\vr(w)=(r_1(w),\ldots,r_{n}(w))$.
\end{defn}

\begin{ex}\label{first_example}
Let 
\begin{equation*}
w=
\begin{bmatrix}
1 & 2 & 3 & 4 & 5 & 6 & 7\\
3 & 7 & 1 & 2 & 5 & 4 & 6
\end{bmatrix}
\in S_7. 
\end{equation*}
This permutation will appear in a running example throughout the paper. We have 
\begin{equation*}
w^{-1}=
\begin{bmatrix}
1 & 2 & 3 & 4 & 5 & 6 & 7\\
3 & 4 & 1 & 6 & 5 & 7 & 2
\end{bmatrix}
\in S_7
\end{equation*}
and $\vr(w)=(3,3,1,3,2,2,1)$.
\end{ex}

The vector $\vr(w)$ is, after subtracting the vector $(1,\ldots,1)$, known as the \emph{Lehmer code} of the \emph{inverse} permutation $w^{-1}$. In this way, permutations $w$ are clearly in bijection with integer vectors $\vr$ of length $n$ with $r_j\in[1,n+1-j]$. It will be convenient to pass between these two indexings of elements of $S_{n}$. We will often write $r_j$ and $\vr$ instead of $r_j(w)$ and $\vr(w)$ when the permutation $w$ has been fixed.

\begin{defn}
The \emph{length} $\ell(w)$ of a permutation $w\in S_n$ is the minimal number of simple transpositions $(i,i+1)$ which need to be composed to obtain $w$. Equivalently, the length $\ell(w)$ is equal to the number of \emph{descents} of $w$, that is, pairs $(i,j)$ for which $i<j$ and $w(i)>w(j)$.

We denote by $w_0\in S_n$ the longest permutation, given by $w_0(i)=n+1-i$. 
\end{defn}

\begin{ex}
The permutation $w\in S_7$ from Example \ref{first_example} has the 8 descents
\begin{equation*}
(1,3),(1,4),(2,3),(2,4),(2,5),(2,6),(2,7),(5,6).
\end{equation*}
Moreover, a minimal decomposition of $w$ into simple transpositions is given by
\begin{equation*}
w=(67)(56)(45)(56)(23)(34)(12)(23)
\end{equation*}
so we have $\ell(w)=8$.
\end{ex}

\begin{defn}
Let $\iota:S_{n-1}\hookrightarrow S_{n}$ be the inclusion sending $w:[n-1]\to[n-1]$ to the permutation $\iota(w):[n]\to[n]$ with $\iota(w)(n)=n$ and $\iota(w)(i)=w(i)$ for $i=1,\ldots,n-1$.

Let $\overline{\iota}:S_{n-1}\hookrightarrow S_{n}$ be the inclusion sending $w:[n-1]\to[n-1]$ to the permutation $\iota(w):[n]\to[n]$ with $\iota(w)(1)=1$ and $\iota(w)(i)=w(i-1)+1$ for $i=2,\ldots,n$.
\end{defn}

\begin{ex}\label{ex:iota}
Let $w\in S_7$ be the permutation from Example \ref{first_example}. Then, we have
\begin{equation*}
\iota(w)=
\begin{bmatrix}
1 & 2 & 3 & 4 & 5 & 6 & 7 & 8\\
3 & 7 & 1 & 2 & 5 & 4 & 6 & 8
\end{bmatrix}
\end{equation*}
and
\begin{equation*}
\overline{\iota}(w)=
\begin{bmatrix}
1 & 2 & 3 & 4 & 5 & 6 & 7 & 8\\
1 & 4 & 8 & 2 & 3 & 6 & 5 & 7
\end{bmatrix}
.
\end{equation*}
\end{ex}

\subsection{Schubert and Richardson varieties}

Fix an $n$-dimensional vector space $\bC^n$, and let $\Fl(n)$ be the space of complete flags $\cL=(0\subset L_1\subset\cdots\subset L_{n-1}\subset \bC^n)$. An additive basis of the cohomology $H^{*}(\Fl(n))$ is given by the classes of Schubert varieties, defined as follows. Let $F=(0\subset F_1\subset\cdots\subset F_{n-1}\subset \bC^n)$ be a fixed flag, and let $w\in S_n$ be a permutation.

\begin{defn}
We define the \emph{Schubert variety} $\Sigma^F_{w}\subset\Fl(n)$ to be the locus of flags $\cL\in\Fl(n)$ satisfying
\begin{equation*}
\dim(L_i\cap F_{n-j})\ge\#\left(\{w(1),\ldots,w(i)\}\cap\{j+1,\ldots,n\}\right).
\end{equation*}
The class $[\Sigma^F_{w}]\in H^{2\ell(w)}(\Fl(n))$, which does not depend on $F$, is denoted $\sigma_w$.
\end{defn}

\begin{ex}\label{ex:schubert}
Consider the permutation
\begin{equation*}
\iota(w)=
\begin{bmatrix}
1 & 2 & 3 & 4 & 5 & 6 & 7 & 8\\
3 & 7 & 1 & 2 & 5 & 4 & 6 & 8
\end{bmatrix}
\in S_8
\end{equation*}
from Example \ref{ex:iota}. Then, the Schubert variety $\Sigma^F_{\iota(w)}\subset\Fl(8)$ is defined by the conditions
\begin{align*}
\dim(L_2\cap F_6)&\ge 2\\
\dim(L_2\cap F_2)&\ge 1\\
\dim(L_5\cap F_4)&\ge 2.
\end{align*}
All other constraints on $\dim(L_i\cap F_{n-j})$ are either immediate or follow from the above inequalities. The inequality $\dim(L_2\cap F_6)\ge 2$ imposes 4 conditions on $\cL$, the inequality $\dim(L_2\cap F_2)\ge1$ imposes 3 additional conditions, and the inequality $\dim(L_5\cap F_4)\ge2$ imposes 1 further condition. In total, we these conditions cut out a closed subvariety of $\Fl(8)$ of codimension $\ell(w)=8$.

A generic element $\cL\in \Sigma^F_{\iota(w)}\subset\Fl(8)$ may be represented by an $8\times 8$ matrix
\begin{equation*}
A_F=
\begin{bmatrix}
0 & 0 & * & * & 0 & * & * & * \\
0 & 0 & * & * & 0 & * & * & * \\
* & 0 & * & * & 0 & * & * & * \\
* & 0 & * & * & 0 & * & * & * \\
* & 0 & * & * & * & * & * & * \\
* & 0 & * & * & * & * & * & * \\
* & * & * & * & * & * & * & * \\
* & * & * & * & * & * & * & * 
\end{bmatrix}
\end{equation*}
where the first $i$ columns of $A$ are a basis of $L_i$, and $F_j\subset\bC^8$ is the subspace of column vectors where the first $8-j$ coordinates vanish. The symbol $*$ denotes a generic complex number.
\end{ex}

Let $F'=(0\subset F'_1\subset\cdots\subset F'_{n-1}\subset \bC^n)$ be a second fixed flag, transverse to $F$, in the sense that $\dim(F_j\cap F'_k)=\max(0,j+k-n)$ for all $j,k$.  Let $w'\in S_n$ be a second permutation.

\begin{defn}
The intersection $\Sigma^F_{w}\cap \Sigma^{F'}_{w'}$ is called a \emph{Richardson variety}.
\end{defn}

The following facts are standard: Schubert varieties are irreducible and reduced of codimension $\ell(w)$, and Richardson varieties are, when non-empty, irreducible and reduced of codimension $\ell(w)+\ell(w')$.

\subsection{Flag matroids and rank polytopes}

In this paper, we deal only with realizable flag matroids, which come from complete flags in $\Fl(n)$. See \cite{cdms} for a survey on flag matroids and their associated polytopes.

Fix as before $n$-dimensional vector space $\bC^n$, and fix in addition a basis $\langle e_1,\ldots,e_n\rangle$. Let $A$ be a non-singular $n\times n$ matrix. The columns of $A$ define a complete flag $\cL(A)$ in $\bC^n$ by taking $L_i$ to be the span of the first $i$ column vectors of $A$.

\begin{defn}
For any $S\subset[n]$ and $j\in[n]$, let $A_{S,j}$ denote the sub-matrix obtained by taking the rows of $A$ indexed by $S$ and the first $j$ columns of $A$. Then, the \emph{flag matroid} associated to $A$ is the data of the \emph{rank function} $\rk_A:\cP([n])\to \bZ_{\ge0}$ defined by 
\begin{equation*}
\rk_A(S)=\sum_{j=1}^{n-1}\rk(A_{S,j}).
\end{equation*}
\end{defn}

\begin{ex}
Let $A=A_\cL$ be the matrix from Example \ref{ex:schubert} and let $S=\{1,2,3,4\}$. Then, we have
\begin{equation*}
\rk_A(S)=1+1+2+3+3+4+4=18.
\end{equation*}
If $A$ were instead a generic matrix, then it would be the case that
\begin{equation*}
\rk_A(S)=1+2+3+4+4+4+4=22.
\end{equation*}
\end{ex}

We will often abuse terminology, identifying the matrix $A$ with its associated flag and flag matroid.

\begin{defn}
Let $A$ be a non-singular $n\times n$ matrix. Then, the \emph{flag matroid polytope} $\FM(A)\subset\bR_{\ge0}^n$ is defined to be the locus cut out by the equation
\begin{equation*}
z_1+\cdots+z_n=\frac{n(n-1)}{2}
\end{equation*}
and the inequalities
\begin{equation*}
z_S:=\sum_{i\in S}z_i\le\rk_A(S)
\end{equation*}
for any subset $S\subset\{1,2,\ldots,n\}$.

More generally, given a sequence $\lambda$ of real numbers $\lambda_1\ge\lambda_2\ge\cdots\ge\lambda_{n}\ge0$, the \emph{weighted flag matroid polytope} $\FM(\lambda,A)\subset\bR_{\ge0}^n$ is defined to be the locus of vectors $\vz=(z_1,\ldots,z_n)$ cut out by the equation
\begin{equation*}
z_1+\cdots+z_n=\lambda_1+\cdots+\lambda_{n}
\end{equation*}
and the inequalities
\begin{equation*}
z_S\le \sum_{j=1}^{n}(\lambda_{j}-\lambda_{j+1})\rk(A_{S,j}),
\end{equation*}
for any subset $S\subset\{1,2,\ldots,n\}$, where by convention we set $\lambda_{n+1}=0$. Taking $\lambda=(n-1,n-2,\ldots,0)$ recovers the definition of $\FM(A)$.
\end{defn}
The required upper bound on $z_S$ may be re-written as
\begin{equation*}
z_S\le \sum_{j=1}^{n}(\rk(A_{S,j})-\rk(A_{S,j-1}))\lambda_{j},
\end{equation*}
where the coefficient in front of $\lambda_{j}$ is 1 if adding the $j$-th column to $A_{S,j-1}$ to obtain $A_{S,j}$ increases the rank, and 0 otherwise.

\begin{defn}
Suppose that $A$ has the property that $A_{S,j}$ has maximal rank for any $S,j$. Then, $A$ is said to define the \emph{uniform} flag matroid.

The \emph{permutohedron} $\Perm(n)\subset\bR^n$ is the flag matroid polytope $\FM(A)$ associated to the uniform flag matroid. Similarly, the \emph{weighted permutohedron} $\Perm(\lambda)\subset\bR^n$ is the weighted flag matroid polytope $\FM(\lambda,A)$ for the uniform flag matroid.
\end{defn}

Note that a generic matrix $A$ defines the uniform flag matroid. The permutohedron is equivalently the convex hull of the points $(w(1)-1,\ldots,w(n)-1)$, where $w$ ranges over all permutations in $S_n$. Similarly, the weighted permutohedron is the convex hull of the points $(\lambda_{w(1)},\ldots,\lambda_{w(n)})$ for any $\lambda$. For any $A$, we have $\FM(\lambda,A)\subset\Perm(\lambda)$.

\begin{ex}
Consider the case $n=3$, which will be a running example throughout the paper. The permutohedron $\Perm(3)\subset\bR^3$ is a regular hexagon. Its projection to the plane given by the first two coordinates is depicted below.
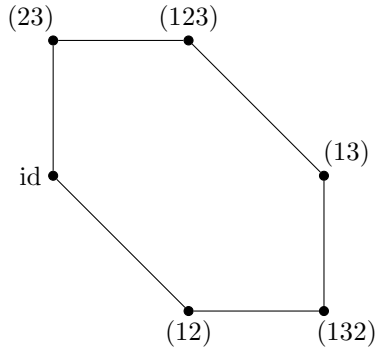
\begin{center}
    \begin{tikzpicture}[scale=0.6,font=\footnotesize]
        \tikzset{
        solid node/.style={circle,draw,inner sep=1.2,fill=black},
        hollow node/.style={circle,draw,inner sep=1.2}
        }
        
        \node[solid node](1) at (0,3) {};
        \node at (-0.5,3) {$\id$};
        \node[solid node](2) at (0,6) {};
        \node at (-0.5,6.5) {$(23)$};
        \node[solid node](3) at (3,6) {};
        \node at (3,6.5) {$(123)$};
        \node[solid node](4) at (6,3) {};
        \node at (6.5,3.5)  {$(13)$};
        \node[solid node](5) at (6,0) {};
        \node at (6.5,-0.5) {$(132)$};
        \node[solid node](6) at (3,0) {};
        \node at (3,-0.5) {$(12)$};
        
         \path (1) edge (2);
         \path (2) edge (3);
         \path (3) edge (4);
         \path (4) edge (5);
         \path (5) edge (6);
         \path (6) edge (1);
       
    \end{tikzpicture}
\end{center}
\end{ex}

\subsection{The moment map}\label{sec:moment}

The main references for this section are the work of Gel'fand-Serganova \cite{gs} and Kapranov \cite{kap}.

Let 
\begin{equation*}
p:\Fl(n)\hookrightarrow\prod_{r=1}^{n-1}\Gr(r,n)\hookrightarrow\prod_{r=1}^{n-1}\bP^{\binom{n}{r}-1}
\end{equation*}
be the Pl\"{u}cker embedding. Let 
\begin{equation*}
\mu_r:\bP^{\binom{n}{r}-1}\to\bR^{n}
\end{equation*}
be the map defined by
\begin{equation*}
\mu_r([x_I])=\frac{\sum_I|x_I|^2 e_I}{\sum_{I}|x_I|^2}
\end{equation*}
where $I$ ranges over all $r$-element subsets of $[n]$, and $e_I$ is the vector $\sum_{i\in I}z_i\in\bR^n$. Then, the \emph{moment map} $\mu:\Fl(n)\to\bR^{n}$ is the composition of $p$ with the sum of the maps $\mu_r$, after projection from $\prod_{r=1}^{n-1}\bP^{\binom{n}{r}-1}$. 

The key property of $\mu$ is the following. Let $A$ be a non-singular $n\times n$ matrix, let $\cL(A)\in\Fl(n)$ be the associated flag, and let $Z_A$ be the $T$-orbit closure of $\cL(A)$. Then, the image of $Z_A$ under $\mu$ is equal to the flag matroid polytope $\FM(A)$. Moreover, the dimension of $Z_A$ as a subvariety of $\Fl(n)$ is equal to the dimension of $\FM(A)$ as a polytope.

Toric degenerations of $Z_A$ inside $\Fl(n)$ correspond to flag matroidal polyhedral subdivisions of $\FM(A)$. More precisely, consider a 1-parameter toric degeneration of $Z_A$ with irreducible components $Z_1,\ldots,Z_m$ on the special fiber. Then, the $Z_i$ are all reduced, and equal to orbit closures of flags $\cL_i\in\Fl(n)$, whose flag matroid polytopes $\FM(A_i)$ (where $A_i$ is obtained by choosing appropriate bases for the components of the flags $\cL_i$) form a polyhedral subdivision of $\FM(A)$.

\section{The HHMP decomposition}\label{sec:hhmp}

In this section, we review the decomposition of $\Perm(\lambda)$ given by Harada-Horiguchi-Masuda-Park \cite{hhmp}.

\begin{defn}
Fix a sequence $\lambda$ of real numbers $\lambda_1\ge\cdots\ge\lambda_{n}\ge0$. The \emph{Gelfand-Zetlin (GZ) polytope} $\GZ(\lambda)$ is defined as follows. Consider the diagram
\begin{equation*}
    \begin{array}{cccccc}
        \lambda_1&  x_{1,2}    &    x_{1,3}   &     \cdots    & x_{1,n-1}	&x_{1,n}\\
                         &\lambda_2&  x_{2,3}       &     \ddots  &x_{2,n-1} 	&x_{2,n} \\
                         &                 &  \ddots	      &		\ddots    &\vdots	      &\vdots\\
                         &                 &                  &\lambda_{n-2}&x_{n-2,n-1}	 &x_{n-2,n}\\
                         &                 &                  &                      &\lambda_{n-1}&x_{n-1,n}\\
                         &                 &                  &                      &                      &\lambda_n
    \end{array}
\end{equation*}
where we also set $x_{i,i}=\lambda_i$ for $i=1,2,\ldots,n$.

Then, $\GZ(\lambda)\subset\bR^{n(n-1)/2}$ is defined to be the subset of vectors $\vx=(x_{k,\ell})_{1\le k<\ell\le n}$ for which any three variables appearing in the configuration 
\begin{equation*} \label{eq:abc}
    \begin{array}{cc}
        a& b\\
          &c
    \end{array}
\end{equation*}
satisfy $a \ge b \ge c$.
\end{defn}

\begin{defn}
Let $w\in S_{n-1}$ be a permutation and write $\vr=\vr(w)=(r_1,\ldots,r_{n-1})$ for the corresponding vector. Define the face $\GZ(\lambda,w)\subset\GZ(\lambda)$ to be the subset of $\GZ(\lambda)$ of points satisfying the equations 
\begin{equation*}\label{eq: perm eq}
x_{i, i+j} = x_{i, i+j-1} \textup{ for } i \in [1,r_j-1], \quad \textup{ and } \quad 
x_{i, i+j} = x_{i+1, i+j} \textup{ for } i\in [r_j +1, n-j].
\end{equation*}
for all $j=1,2,\ldots,n-1$.
\end{defn}
The key property of $\GZ(\lambda,w)$ is the following. Suppose that the coordinates $x_{i,i+j-1}$ are given for some fixed $j\in[1,n-1]$ and all $i=1,2,\ldots,n-j+1$. Then, all but one of the entries $x_{i,i+j}$ is determined by the above equations; the unique entry which is not is $x_{r_j,r_j+j}$ (this entry exists because $r_j\in[1,n-j]$), which is constrained by the inequality
\begin{equation*}
x_{r_j+1,r_j+j}\le x_{r_j,r_j+j}\le x_{r_j,r_j+j-1}.
\end{equation*}
Because the entries $x_{i,i}=\lambda_i$ are fixed, the dimension of any face $\GZ(\lambda,w)$ is easily seen to be equal to $n-1$.

\begin{ex}\label{face_diagram}
Consider the permutation from Example \ref{first_example}
\begin{equation*}
w=
\begin{bmatrix}
1 & 2 & 3 & 4 & 5 & 6 & 7\\
3 & 7 & 1 & 2 & 5 & 4 & 6
\end{bmatrix}
\in S_7.
\end{equation*}
with $\vr(w)=(3,3,1,3,2,2,1)$. Then, we represent the polytope $\GZ(\lambda,w)$ by the \emph{face diagram} below, where entries $x_{i,j}$ constrained to be equal are connected by an edge.
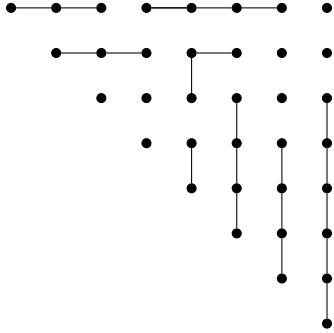
\begin{center}
    \begin{tikzpicture}[scale=0.6,font=\footnotesize]
        \tikzset{
        solid node/.style={circle,draw,inner sep=1.2,fill=black},
        hollow node/.style={circle,draw,inner sep=1.2}
        }
        \node[solid node](11) at (0,7) {};
        \node[solid node](12) at (1,7) {};
        \node[solid node](13) at (2,7) {};
        \node[solid node](14) at (3,7) {};
        \node[solid node](15) at (4,7) {};
        \node[solid node](16) at (5,7) {};
        \node[solid node](17) at (6,7) {};
        \node[solid node](18) at (7,7) {};
        
        \node[solid node](22) at (1,6) {};
        \node[solid node](23) at (2,6) {};
        \node[solid node](24) at (3,6) {};
        \node[solid node](25) at (4,6) {};
        \node[solid node](26) at (5,6) {};
        \node[solid node](27) at (6,6) {}; 
        \node[solid node](28) at (7,6) {};
        
        \node[solid node](33) at (2,5) {};
        \node[solid node](34) at (3,5) {};
        \node[solid node](35) at (4,5) {};
        \node[solid node](36) at (5,5) {};
        \node[solid node](37) at (6,5) {}; 
        \node[solid node](38) at (7,5) {};
        
        \node[solid node](44) at (3,4) {};
        \node[solid node](45) at (4,4) {};
        \node[solid node](46) at (5,4) {};
        \node[solid node](47) at (6,4) {}; 
        \node[solid node](48) at (7,4) {};
        
        \node[solid node](55) at (4,3) {};
        \node[solid node](56) at (5,3) {};
        \node[solid node](57) at (6,3) {}; 
        \node[solid node](58) at (7,3) {};
        
        \node[solid node](66) at (5,2) {};
        \node[solid node](67) at (6,2) {}; 
        \node[solid node](68) at (7,2) {};
        
        \node[solid node](77) at (6,1) {}; 
        \node[solid node](78) at (7,1) {};
        
        \node[solid node](88) at (7,0) {};
        
         \path (11) edge (12);
         \path (12) edge (13);
         \path (14) edge (15);
         \path (15) edge (16);
         \path (16) edge (17);
         \path (14) edge (15);
         
         \path (22) edge (23);
         \path (23) edge (24);
         \path (25) edge (26);
         
         \path (25) edge (35);
         
          \path (45) edge (55);
          
          \path (36) edge (46);
         \path (46) edge (56);
         \path (56) edge (66);
         
         \path (47) edge (57);
         \path (57) edge (67);
          \path (67) edge (77);
          
           \path (38) edge (48);
         \path (48) edge (58);
         \path (58) edge (68);
         \path (68) edge (78);
         \path (78) edge (88);
        
    \end{tikzpicture}
\end{center}
The coordinates $x_{ii}$ for $i=1,\ldots,8$ are constrained to be equal to $\lambda_i$. For $j=1,\ldots,7$, we see that all coordinates $x_{i,i+j}$ are determined by one of the coordinates $x_{i,i+j-1},x_{i+1,i+j}$ appearing on the diagonal to the immediate southwest, except when $i=r_j$. In the example above, the undetermined entries are $x_{34},x_{35},x_{14},x_{37},x_{27},x_{28},x_{18}$. These $n-1=7$ entries correspond to the dimension of $\GZ(\lambda,w)$.
\end{ex}

\begin{defn}
Define the map $\Phi:\GZ(\lambda)\to\Perm(\lambda)$ as follows. For $\vx\in\GZ(\lambda)$ and $j=0,1,\ldots,n-1$, write
\begin{equation*}
y_j=x_{1,1+j}+\cdots+x_{n-j,n}.
\end{equation*}
Then, define 
\begin{equation*}
\Phi(\vx)=(y_0-y_1,\ldots,y_{n-2}-y_{n-1},y_{n-1}).
\end{equation*}
\end{defn}

\begin{ex}\label{hmmp_n=3}
Take $n=3$ and $\lambda=(2,1,0)$. Then, we have
\begin{align*}
\GZ(\lambda,\id)&=\{(x_{12},x_{23},x_{13})\mid x_{12}\in [1,2], x_{23}=0, x_{13}\in[0,x_{12}]\},\\
\GZ(\lambda,(12))&=\{(x_{12},x_{23},x_{13})\mid x_{12}=2, x_{23}\in[0,1], x_{13}\in[x_{23},2]\}.
\end{align*}
We compute that 
\begin{equation*}
\Phi(\vx)=
\begin{cases}
(3-x_{12},x_{12}-x_{13},x_{13}) &\text{ if }\vx\in \GZ(\lambda,\id),\\
(1-x_{23},2+x_{23}-x_{13},x_{13}) &\text{ if }\vx\in \GZ(\lambda,(12)).
\end{cases}
\end{equation*}
The images of $\GZ(\lambda,\id)$ and $\GZ(\lambda,(12))$ in $\Perm(3)\subset\bR^3$ are thus the sub-polygons where the first coordinate is at least and at most 1, respsectively. These are depicted in the figure below.
\begin{center}
    \begin{tikzpicture}[scale=0.6,font=\footnotesize]
        \tikzset{
        solid node/.style={circle,draw,inner sep=1.2,fill=black},
        hollow node/.style={circle,draw,inner sep=1.2}
        }
        \node[solid node](1) at (0,3) {};
        \node at (-0.5,3) {$\id$};
        \node[solid node](2) at (0,6) {};
        \node at (-0.5,6.5) {$(23)$};
        \node[solid node](3) at (3,6) {};
        \node at (3,6.5) {$(123)$};
        \node[solid node](4) at (6,3) {};
        \node at (6.5,3.5)  {$(13)$};
        \node[solid node](5) at (6,0) {};
        \node at (6.5,-0.5) {$(132)$};
        \node[solid node](6) at (3,0) {};
        \node at (3,-0.5) {$(12)$};
        
         \path (1) edge (2);
         \path (2) edge (3);
         \path (3) edge (4);
         \path (4) edge (5);
         \path (5) edge (6);
         \path (6) edge (1);
         \path (3) edge (6);
         
         \node at (1.5,3.5) {\tiny{$\GZ(\lambda,(12))$}};
         \node at (4.5,2.5) {\tiny{$\GZ(\lambda,\id)$}};
       
    \end{tikzpicture}
\end{center}
\end{ex}

\begin{ex}
Consider the restriction of $\Phi$ to the face $\GZ(\lambda,w)$, where $w$ is as in Example \ref{face_diagram}. We compute
\begin{align*}
\Phi(\vx)=(x_{33}+x_{44}-x_{34},&x_{34}+x_{45}-x_{35},x_{13}+x_{24}-x_{14},x_{36}+x_{47}-x_{37},\\
&x_{26}+x_{37}-x_{27},x_{27}+x_{38}-x_{28},x_{17}+x_{28}-x_{18}).
\end{align*}
Note that the negative term in the $j$-th coordinate of $\Phi(\vx)$ is equal to $x_{r_j,r_j+j}$, the unique entry in the $j$-th diagonal not determined by those in the previous diagonal. Note also that the first coordinate is bounded between $\lambda_4=x_{44}$ and $\lambda_3=x_{33}$, and more generally each coordinate is bounded between its two positive terms.
\end{ex}

\begin{thm}\cite[Proposition 5.2]{hhmp}
The map $\Phi$ is a bijection upon restriction to the union of faces $\GZ(\lambda,w)$, where $w$ ranges over all permutations in $S_{n-1}$.
\end{thm}

The images of the $\GZ(\lambda,w)$ in $\Perm(\lambda)$, which we abusively denote $\GZ(\lambda,w)$, therefore give a polyhedral decomposition of $\Perm(\lambda)$, which we refer to as the \emph{HHMP decomposition}.

We recall the construction of the inverse map $\Perm(\lambda)\to\bigcup_{w\in S_{n-1}}\GZ(\lambda,w)$. Suppose that $\vz=(z_1,\ldots,z_n)\in\Perm(\lambda)$. Then, there exists an integer $r_1\in[1,n-1]$ (not necessarily unique) for which $z_1\in[\lambda_{r_1+1},\lambda_{r_1}]$, and thus
\begin{equation*}
\lambda_{r_1+1}\le (\lambda_{r_1+1}+\lambda_{r_1})-z_1 \le \lambda_r.
\end{equation*}
Set $x_{i,i+1}=x_{i,i}=\lambda_i$ for $i\in[1,r_1-1]\cup[r_1+1,n-1]$, and set $x_{i,i+1}=(\lambda_{r_1+1}+\lambda_{r_1})-z_1$. Now, set $\lambda'=(x_{1,2},\ldots,x_{n-1,n})$, and iterate this procedure. Namely, define $x_{i,i+j}$ with $j>1$ inductively via the inverse map $\Perm(\lambda')\to\bigcup_{w'\in S_{n-2}}\GZ(\lambda',w')$.

\section{The toric degeneration}\label{sec:explicit_degen}

We now come to our main construction, giving an explicit toric degeneration of the generic torus orbit closure $Z\subset \Fl(n)$ into a union of special orbit closures $Z_w\subset\Fl(n)$.

Strictly speaking, the construction of this section does strictly less than this; it will only follow from our computations of flag matroid polytopes in the next section that our construction gives the desired degeneration. We more precisely describe a sequence of degenerations of irreducible subschemes of $\Fl(n)$, starting with $Z$. In each degeneration, we will identify two distinct irreducible subschemes of their flat limits, some of which will be ignored. We iterate the construction until reaching the subschemes $Z_w\subset\Fl(n)$. Note that $\dim(Z)=n-1$, and that the dimensions of the limit subschemes can only go down at each step.

We will find in the next section that in fact, the $Z_w$ also have dimension $n-1$, which implies that all of the intermediate subschemes in the sequence of degenerations from $Z$ to $Z_w$ also have dimension $n-1$. Therefore, all of these subschemes are components of the flat limits of the corresponding degenerations. Moreover, we will find, by the existence of the HMMP decomposition, that no components other than the $Z_w$ can appear in the end, which implies that the limit components we identify in each of the intermediate degenerations are the only ones. However, we emphasize that we will not need any of these conclusions in the construction of this section.

\subsection{The first row}
Let 
\begin{equation*}
A=
\begin{bmatrix}
a_{1,1} & \cdots  & a_{1,n}\\
 & \vdots &  \\
 a_{n,1} & \cdots & a_{n,n}
\end{bmatrix}
\end{equation*}
be a generic $n\times n$ matrix, which we identify with its corresponding flag $\cL(A)\in\Fl(n)$. More generally, for $j=0,1,\ldots,n-1$, let 
\begin{equation*}
A^j=
\begin{bmatrix}
0 & \cdots & 0 & a_{1,j+1} & a_{1,j+2} & \cdots & a_{1,n}\\
a_{2,1} & \cdots & a_{2,j} & a_{2,j+1} & a_{2,j+2} & \cdots & a_{2,n}\\
 & &  \vdots & & & & \\
 a_{n,1} & \cdots  & a_{n,j} & a_{n,j+1} & a_{n,j+2} & \cdots & a_{n,n}
\end{bmatrix}
\end{equation*}
be an $n\times n$ matrix which is zero in the left-most $j$ entries of its first row, but generic otherwise. (We have $A^0=A$.)

Fix now $j\in [1,n-1]$, and consider a degeneration in which the $j$-th entry in the first row of $A^{j-1}$ is sent to zero. More precisely, consider the matrix
\begin{equation*}
A^{j-1}_t=
\begin{bmatrix}
0 & \cdots & 0 & t & a_{1,j+1} & a_{1,j+2} & \cdots & a_{1,n}\\
a_{2,1} & \cdots & a_{2,j-1} & a_{2,j} & a_{2,j+1} & a_{2,j+2} & \cdots & a_{2,n}\\
 & & & \vdots & & & & \\
 a_{n,1} & \cdots & a_{n,j-1} & a_{n,j} & a_{n,j+1} & a_{n,j+2} & \cdots & a_{n,n}
\end{bmatrix}
\end{equation*}
whose entries are taken in $\bC[[t]]$, and its corresponding $T$-orbit closure $Z^{j-1}_t$ in $\Fl(n)$. 
The family of $T$-orbit closures $Z^{j-1}_t\subset\Fl(n)\times\Spec\bC[[t]]$, where $t\neq 0$, is clearly $T$-equivariant. It is also clear that the $T$-orbit closure $Z^j\subset\Fl(n)$ of $A^{j}$ appears in the flat limit of $Z^{j-1}_t$ as $t\to0$.

Consider now the 1-parameter family of matrices
\begin{equation*}
\begin{bmatrix}
0 & \cdots & 0 & t & a_{1,j+1} & a_{1,j+2} & \cdots & a_{1,n}\\
a_{2,1}t & \cdots & a_{2,j-1}t & a_{2,j}t & a_{2,j+1}t & a_{2,j+2}t & \cdots & a_{2,n}t\\
 & & & \vdots & & & & \\
 a_{n,1}t & \cdots & a_{n,j-1}t & a_{n,j}t & a_{n,j+1}t & a_{n,ij2}t & \cdots & a_{n,n}t
\end{bmatrix}
\in Z^{j-1}_t,
\end{equation*}
obtained by acting by $t\in T$ in the last $n-1$ rows of $A_t^{j-1}$. The limit of the corresponding flags in $\Fl(n)$ is the flag defined by the matrix
\begin{equation*}
A^j_{+}:=
\begin{bmatrix}
0 & \cdots & 0 & 1 & 1 & 0 & \cdots & 0\\
a_{2,1} & \cdots & a_{2,j-1} & a_{2,j} & 0 & a_{2,j+2} & \cdots & a_{2,n}\\
 & & & \vdots & & & \\
 a_{n,1} & \cdots & a_{n,j-1} & a_{n,j} & 0 & a_{n,j+2} & \cdots & a_{n,n}
\end{bmatrix}
.
\end{equation*}
Indeed, the first $j$ columns are obtained by dividing by $t$, the $(j+1)$-th column is obtained by substituting $t=0$ and re-scaling, and the remaining columns are obtained by subtracting the appropriate multiple of the $(j+1)$-th column, and then dividing by $t$. The matrix $A^j_{+}$ is non-singular by the genericity assumption.

Let $Z^j_+$ be the $T$-orbit closure of $A^j_{+}$ in $\Fl(n)$. Then, $Z^j_+$ is also a subscheme of the flat limit of $Z^{j-1}_t$ as $t\to 0$. Therefore, the flat limit of $Z^{j-1}_t$ as $t\to 0$ contains the irreducible subschemes $Z^{j},Z^j_+$. In the case $j=n-1$, we throw out the subscheme $Z^{n-1}$. We may safely do this because there are two torus factors acting trivially on $Z^{n-1}$, namely the factor acting on the first row of $A^{n-1}$ and the factor acting diagonally on the last $n-1$ rows, hence $\dim(Z^{n-1})=n-2$. (In fact, $Z^{n-1}\subset Z^{n-1}_+$.) On the other hand, when $j<n-1$, note that $Z^{j}\neq Z^j_+$, because a general point of $Z^j_+$ has 
\begin{equation*}
\begin{bmatrix}
1 \\
0\\
\vdots \\
0
\end{bmatrix}
\in L_{j+1}
\end{equation*}
but a general point of $Z^j$ does not.

To summarize, we have:
\begin{prop}
For all $j=1,2,\ldots,n-2$, there exists a toric degeneration of the orbit closure $Z^{j-1}\subset \Fl(n)$ of the flag determined by the matrix $A^{j-1}$, whose special fiber contains the distinct subschemes $Z^{j}\subset\Fl(n)$ and $Z^j_+\subset\Fl(n)$. For $j=n-1$, the special fiber contains the subscheme $Z^{n-1}_+\subset\Fl(n)$.
\end{prop}

Thus, starting from the generic torus orbit closure $Z=Z^0$, we perform $n-1$ degenerations in total, obtaining in the end the collection of subschemes $Z^1_+,\ldots,Z^{n-1}_+$. As explained above, it will turn out that the $Z^j_+$ all have dimension $n-1$, which implies a posteriori that they appear with multiplicity 1 as components of the flat limits of their corresponding degenerations. It will furthermore turn out that we have identified all of the components of the flat limits of the degenerations, so in particular that $[Z]$ is equal to the sum of the classes $[Z^j_+]$. However, we will continue not to assume this in what follows.

\subsection{Iterating over rows}
For $j=1,2,\ldots,n-1$, we repeat the degeneration on the matrix $A_j^{+}$. More precisely, consider the $(n-1)\times (n-1)$ matrix obtained by ignoring the first row and $(j+1)$-st column of $A_j^{+}$. Then, we send, from left to right, the entries in the top row of this new matrix to zero, extracting two subschemes of the flat limit, except in the $(n-2)$-nd step, in which case we only extract one. We continue this procedure until reaching the last row. 

We now describe the matrices whose orbit closures are reached at the end of this process. Let $w\in S_{n-1}$ be a permutation. 
\begin{defn}
Define the matrix $A_w$ in the following steps. 
\begin{enumerate}
\item For $i=1,2,\ldots,n-1$, do the following:
\begin{itemize}
\item Place the symbol $\star_2$ in row $i$ and column $w^{-1}(i)+1$. 
\item Place the symbol $\star_1$ in row $i$ and the right-most column that lies to the left of column $w^{-1}(i)+1$, but does not contain the symbol $\star_2$ in any row above row $i$. Such a column exists, because the symbol $\star_2$ is never placed in the first column of $A_w$ for $i=1,2,\ldots,n-1$.
\end{itemize}
\item Place the symbol $\star_2$ in row $n$ and column 1. 
\item Place a 0 in all remaining entries of $A_w$. 
\item Replace all symbols $\star_1,\star_2$ in $A_w$ with generic complex numbers. 
\end{enumerate}
\end{defn}

\begin{ex}\label{end_matrix}
Let
\begin{equation*}
w=
\begin{bmatrix}
1 & 2 & 3 & 4 & 5 & 6 & 7\\
3 & 7 & 1 & 2 & 5 & 4 & 6
\end{bmatrix}
,w^{-1}=
\begin{bmatrix}
1 & 2 & 3 & 4 & 5 & 6 & 7\\
3 & 4 & 1 & 6 & 5 & 7 & 2
\end{bmatrix}
.
\end{equation*}
Then, we obtain the below matrix $A_w$, where on the right the symbols $*$ denote generic (and distinct) complex numbers.
\begin{equation*}
\begin{bmatrix}
0 & 0 & \star_1 & \star_2 & 0 & 0 & 0 & 0 \\
0 & 0 & \star_1 & 0 & \star_2 & 0 & 0 & 0 \\
\star_1 & \star_2 & 0 & 0 & 0 & 0 & 0 & 0 \\
0 & 0 & 0 & 0 & 0 & \star_1 & \star_2 & 0 \\
0 & 0 & \star_1 & 0 & 0 & \star_2 & 0 & 0\\
0 & 0 & \star_1 & 0 & 0 & 0 & 0 & \star_2 \\
\star_1 & 0 & \star_2 & 0 & 0 & 0 & 0 & 0\\
\star_2 & 0 & 0 & 0 & 0 & 0 & 0 & 0
\end{bmatrix}
\rightarrow
\begin{bmatrix}
0 & 0 & * & * & 0 & 0 & 0 & 0 \\
0 & 0 & * & 0 & * & 0 & 0 & 0 \\
* & * & 0 & 0 & 0 & 0 & 0 & 0 \\
0 & 0 & 0 & 0 & 0 & * & * & 0 \\
0 & 0 & * & 0 & 0 & * & 0 & 0\\
0 & 0 & * & 0 & 0 & 0 & 0 & * \\
* & 0 & * & 0 & 0 & 0 & 0 & 0\\
* & 0 & 0 & 0 & 0 & 0 & 0 & 0
\end{bmatrix}
\end{equation*}
\end{ex}

Observe that the matrix $A_w$ is invertible, if all entries corresponding to the symbol $\star_2$ are non-zero. Indeed, the columns of $A_w$ may be permuted so that the entries $\star_2$ appear on the diagonal and the resulting matrix is upper triangular. We may therefore define:
\begin{defn}
Let $\cL_w:=\cL(A_w)\in\Fl(n)$ be the flag associated to $A_w$, and let $Z_w:=\overline{T\cdot\cL_w}\subset\Fl(n)$ be the corresponding orbit closure.
\end{defn}

We have proven the following:

\begin{prop}
There exists a sequence of toric degenerations of $Z$ whose special fiber in the final step contains the orbit closures $Z_w$, where $w$ ranges over all permutations in $S_{n-1}$.
\end{prop}

We still do not claim that $Z_w$ appear as \emph{components} of the special fibers, nor that no other components appear. This is proven in the next section.

\begin{ex}\label{n=3_matrices}
For $n=3$, we have
\begin{equation*}
A_{\id}=
\begin{bmatrix}
* & * & 0 \\
*& 0 & *\\
* & 0 & 0
\end{bmatrix}
,
A_{(12)}=
\begin{bmatrix}
0 & * & * \\
* & * & 0\\
* & 0 & 0
\end{bmatrix}
.
\end{equation*}
\end{ex}

\section{GZ faces as flag matroid polytopes}\label{sec:fmp}

To the orbit closures $Z_w\subset \Fl(n)$ appearing in the special fibers of the degenerations described in the previous section, we may associate flag matroid polytopes $\FM(\lambda,w)$. The purpose of this section is to identify these flag matroid polytopes with the polytopes $\GZ(\lambda,w)$ appearing in the HMMP decomposition.

\begin{defn}
Define $\FM(\lambda,w)=\FM(\lambda,A_w)$ to be the flag matroid polytope associated to $A_w$.
\end{defn}

\begin{ex}\label{FM_GZ_3}
Take $n=3$ and $\lambda=(2,1,0)$. We compute $\FM(\lambda,w)\subset\Perm(3)$ for $w=\id,(12)$, using the matrices $A_w$ found in Example \ref{n=3_matrices}. 

When $w=\id$, Then, the only non-trivial rank condition coming from $A_w$ is $z_2+z_3\le 1+1=2$, and because $z_1+z_2+z_3=3$, this is equivalent to $z_1\ge 1$. Comparing to Example \ref{hmmp_n=3}, we find that $\FM(\lambda,\id)=GZ(\lambda,\id)$. 

Similarly, when $w=(12)$, the only non-trivial rank condition is $z_1\le 1$, which gives $\FM(\lambda,(12))=GZ(\lambda,(12))$.
\end{ex}

To show that $\FM(\lambda,w)=\GZ(\lambda,w)$ in general, we will induct on $n$, considering the rank inequalities obtained by restricting to submatrices of $A_w$. We will first need two lemmas.

\begin{lem}\label{rank_lemma_S+}
Let $w\in S_{n-1}$ be a permutation, and let $S\subset\{2,3,\ldots,n\}$ be any subset. Write $S^+=\{1\}\cup S$.

Then, we have
\begin{equation*}
\rk((A_w)_{S^+,j})=
\begin{cases}
\rk((A_w)_{S,j}) &\text{ if } j\le r_1-1,\\
\rk((A_w)_{S,r_1-1})+1  &\text{ if } j=r_1,\\
\rk((A_w)_{S,j})+1 &\text{ if } j\ge r_1+1.
\end{cases}
\end{equation*}
\end{lem}

\begin{proof}
If $j\le r_1-1$, then $(A_w)_{S^+,j}$ is obtained from $(A_w)_{S,j}$ by adding a zero row, so the claim is obvious in this case. In the case $j=r_1$, the matrix $(A_w)_{S^+,r_1}$ is obtained from $(A_w)_{S,r_1-1}$ by first adding a zero row, and then adding a column with non-zero coordinate in the first row, which was previously zero, so the rank increases by 1. Finally, if $j\ge r_1+1$, then $(A_w)_{S^+,j}$ is obtained from $(A_w)_{S,j}$ by adding a row with non-zero coordinate in the $(r_1+1)$-st column, which was previously zero, so the rank increases by 1.
\end{proof}

\begin{lem}\label{rank_lemma_w'}
Let $w\in S_{n-1}$ be a permutation, and let $\vr=(r_1,\ldots,r_{n-1})$ be the corresponding integer vector. Let $w'\in S_{n-2}$ be the permutation corresponding to $\vr'=(r_2,\ldots,r_{n-1})$. Then, the matrix $A_{w'}$ is obtained by deleting the first row and the $(r_1+1)$-st column of $A_w$.

Furthermore, let $S\subset\{2,3,\ldots,n\}$ be any subset. Then, we have
\begin{equation*}
\rk((A_{w'})_{S,j})=
\begin{cases}
\rk((A_w)_{S,j}) &\text{ if } j\le r_1,\\
\rk((A_w)_{S,j+1}) &\text{ if } j\ge r_1.
\end{cases}
\end{equation*}
\end{lem}
Our convention in Lemma \ref{rank_lemma_w'} is that we index the rows of $A_{w'}$ by $2,3,\ldots,n$, but shift the indices on the columns to the right of the deleted one so that they are labelled $1,2,\ldots,n-1$. Note that the case $j=r_1$ is covered by both cases above, and the claim is that
\begin{equation*}
\rk((A_{w'})_{S,r_1})=\rk((A_w)_{S,r_1})=\rk((A_w)_{S,r_1+1}).
\end{equation*}
The equality $\rk((A_w)_{S,r_1})=\rk((A_w)_{S,r_1+1})$ is due to the fact that the $(r_1+1)$-st column of $(A_w)_{S,r_1+1}$ is zero.

\begin{ex}
Let
\begin{equation*}
w=
\begin{bmatrix}
1 & 2 & 3 & 4 & 5 & 6 & 7\\
3 & 7 & 1 & 2 & 5 & 4 & 6
\end{bmatrix}
,w^{-1}=
\begin{bmatrix}
1 & 2 & 3 & 4 & 5 & 6 & 7\\
3 & 4 & 1 & 6 & 5 & 7 & 2
\end{bmatrix}
,
\end{equation*}
so that $\vr=(r_1,\ldots,r_{n-1})=(3,3,1,3,2,2,1)$ and $\vr'=(3,1,3,2,2,1)$. Then, we have
\begin{equation*}
w'=
\begin{bmatrix}
1 & 2 & 3 & 4 & 5 & 6 \\
2 & 6 & 1 & 4 & 3 & 5
\end{bmatrix}
,
(w')^{-1}=
\begin{bmatrix}
1 & 2 & 3 & 4 & 5 & 6 \\
3 & 1 & 5 & 4 & 6 & 2
\end{bmatrix}
\end{equation*}
and
\begin{equation*}
A_w=
\begin{bmatrix}
0 & 0 & * & * & 0 & 0 & 0 & 0 \\
0 & 0 & * & 0 & * & 0 & 0 & 0 \\
* & * & 0 & 0 & 0 & 0 & 0 & 0 \\
0 & 0 & 0 & 0 & 0 & * & * & 0 \\
0 & 0 & * & 0 & 0 & * & 0 & 0\\
0 & 0 & * & 0 & 0 & 0 & 0 & * \\
* & 0 & * & 0 & 0 & 0 & 0 & 0\\
* & 0 & 0 & 0 & 0 & 0 & 0 & 0
\end{bmatrix}
,
A_{w'}=
\begin{bmatrix}
0 & 0 & *  & * & 0 & 0 & 0 \\
* & * & 0  & 0 & 0 & 0 & 0 \\
0 & 0 & 0 & 0 & * & * & 0 \\
0 & 0 & *  & 0 & * & 0 & 0\\
0 & 0 & *  & 0 & 0 & 0 & * \\
* & 0 & *  & 0 & 0 & 0 & 0\\
* & 0 & 0 & 0 & 0 & 0 & 0
\end{bmatrix}
\end{equation*}
We see that $A_{w'}$ is obtained by deleting the first row and 4th column of $A_w$. The rows of $A_{w'}$ are indexed $2,3,\ldots,8$, which correspond to the same-numbered rows of  of $A_{w}$, with the 4th coordinate deleted. The columns of $A_{w'}$ are indexed $1,2,\ldots,7$; the first three columns are the first three columns of $A_w$, with the first coordinate deleted, and columns $4,5,6,7$ of $A_{w'}$ are columns $5,6,7,8$ of $A_w$, with the first coordinate deleted.
\end{ex}

\begin{proof}[Proof of Lemma \ref{rank_lemma_w'}]
That $A_{w'}$ is obtained from $A_w$ by deleting the first row and $(r_1+1)$-st column is immediate by construction, because the integers $r_j$ determine the relative positions of the symbols $\star_2$ in either matrix, which in turn determine the positions of the symbols $\star_1$.

If $j\le r_1$, then $(A_{w'})_{S,j}=(A_w)_{S,j}$, so the ranks of the two matrices are equal. If $j\ge r_1+1$, then $(A_{w'})_{S,j}$ is obtained from $(A_w)_{S,j+1}$ by deleting a zero row, so the ranks of these two matrices are equal.
\end{proof}

\begin{thm}\label{FM=GZ}
For any $\lambda,w$, we have $\FM(\lambda,w)=\GZ(\lambda,w)$.
\end{thm}

\begin{proof}
We induct on $n$. The first interesting case is $n=3$, which we have verified in Example \ref{FM_GZ_3}. Assume the conclusion for $\Fl(n-1)$, and fix $w\in S_{n-1}$ and the corresponding $\vr$. We first show that $\GZ(\lambda,w)\subset\FM(\lambda,w)$. Let $\vz=\Phi(\vx)\in\GZ(\lambda,w)$ be any point. By assumption, we have $x_{i,i+1}=\lambda_i$ for $i=1,2,\ldots,r_1-1$ and $i=r_1+1,\ldots,n-1$, and $x_{r_1,r_1+1}=\lambda_{r_1}+\lambda_{r_1+1}-z_1$.

Write now:
\begin{align*}
\vz'&=(z_2,\ldots,z_n),\\
\vx'&=(x_{i,j})_{1\le i<j+1\le n},\\
\lambda'&=(\lambda_1,\cdots,\lambda_{r_1-1},\lambda_{r_1}+\lambda_{r_1+1}-z_1,\lambda_{r_1+2},\ldots,\lambda_n),\\
\vr'&=(r_2,\ldots,r_{n-1}),
\end{align*}
and furthermore let $w'\in S_{n-2}$ be the permutation corresponding to $\vr'$. Then, we have $\vz'=\Phi(\vx')\in\GZ(\lambda',w')=\FM(\lambda',w')$, by the inductive hypothesis.

The $(n-1)\times(n-1)$ matrix $A_{w'}$ is obtained by deleting the first row and $(r_1+1)$-st column of $A_w$. The $(r_1+1)$-st column of $A_w$ is zero except in the first row. For any $S\subset\{2,3,\ldots,n\}$, we therefore have (again, using the convention that the rows of $A_{w'}$ are labelled $2,\ldots,n$ but the columns are labelled $1,\ldots,n-1$)
\begin{align*}
z_S=z'_S &\le \sum_{j\in[1,n-1]}(\rk((A_{w'})_{S,j})-\rk((A_{w'})_{S,j-1}))\lambda'_j\\
&=\sum_{j\in[1,r_1-1]\cup[r_1+2,n]}(\rk((A_{w})_{S,j})-\rk((A_{w})_{S,j-1}))\lambda_j\\
&\qquad+(\rk((A_{w})_{S,r_1})-\rk((A_{w})_{S,r_1-1}))(\lambda_{r_1}+\lambda_{r_1+1}-z_1),
\end{align*}
where we have applied Lemma \ref{rank_lemma_w'}.

This implies on the one hand that
\begin{align*}
z_S &\le\sum_{j\in[1,r_1]\cup[r_1+2,n]}(\rk((A_{w})_{S,j})-\rk((A_{w})_{S,j-1}))\lambda_j\\
&=\sum_{j\in[1,n]}(\rk((A_{w})_{S,j})-\rk((A_{w})_{S,j-1}))\lambda_j,
\end{align*}
because $\lambda_{r_1}+\lambda_{r_1+1}-z_1\le\lambda_{r_1}$ and $\rk((A_{w})_{S,r_1+1})=\rk((A_{w})_{S,r_1})$.

On the other hand, by Lemma \ref{rank_lemma_S+}, we have
\begin{align*}
z_S &\le \sum_{j\in[1,r_1-1]\cup[r_1+2,n]}(\rk((A_{w})_{S,j})-\rk((A_{w})_{S,j-1}))\lambda_j\\
&\qquad+(\rk((A_{w})_{S,r_1})-\rk((A_{w})_{S,r_1-1}))(\lambda_{r_1}+\lambda_{r_1+1}-z_1)\\
&\le \sum_{j\in[1,r_1-1]\cup[r_1+2,n]}(\rk((A_{w})_{S^+,j})-\rk((A_{w})_{S^+,j-1}))\lambda_j\\
&\qquad+(\rk((A_{w})_{S^+,r_1+1})-\rk((A_{w})_{S^+,r_1}))\lambda_{r_1+1}+(\lambda_{r_1}-z_1)
\end{align*}
and so
\begin{equation*}
z_{S^+} \le \sum_{j\in[1,n]}(\rk((A_{w})_{S^+,j})-\rk((A_{w})_{S^+,j-1}))\lambda_j,
\end{equation*}
because $\rk((A_{w})_{S^+,r_1})-\rk((A_{w})_{S^+,r_1-1})=1$.
The vector $\vz\in\GZ(\lambda,w)$ therefore satisfies all needed inequalities for $\FM(\lambda,w)$, so we conclude that $\GZ(\lambda,w)\subset\FM(\lambda,w)$.

Conversely, suppose that $\vz\in\FM(\lambda,w)$. In particular, we have
\begin{equation*}
z_1\le \lambda_{r_1},
\end{equation*}
because $(A_w)_{\{1\},j}=0$ for $j\le r_1-1$, and also
\begin{equation*}
z_2+\cdots+z_n\le \sum_{j\in[1,r_1]\cup[r_1+2,n]}\lambda_j,
\end{equation*}
because the $(r_1+1)$-st column of $(A_w)_{{\{2,\ldots,n\}},r_1+1}$ is zero. In particular, we have $z_1\in[\lambda_{r_1+1},\lambda_{r_1}]$.

Define again
\begin{align*}
\vz'&=(z_2,\ldots,z_n),\\
\lambda'&=(\lambda_1,\cdots,\lambda_{r_1-1},\lambda_{r_1}+\lambda_{r_1+1}-z_1,\lambda_{r_1+2},\ldots,\lambda_n),\\
\vr'&=(r_2,\ldots,r_{n-1}),
\end{align*}
and let $w'\in S_{n-2}$ be the corresponding permutation to $\vr'$. It now suffices to show that $\vz'\in\FM(\lambda',w')$, because by the inductive hypothesis, there exists $\vx'\in\GZ(\lambda',w')$ with $\vz'=\Phi(\vx')$, and combining the data of $\vx'$ and $\lambda$ gives a point $\vx\in\GZ(\lambda,w)$ with $\Phi(\vx)=\vz$.

That $\vz'\in\FM(\lambda',w')$ amounts to the requirement that, for all $S\subset\{2,3,\ldots,n\}$, we have
\begin{align*}
z_S &\le \sum_{j\in[1,n-1]}(\rk((A_{w'})_{S,j})-\rk((A_{w'})_{S,j-1}))\lambda'_j\\
&=\sum_{j\in[1,r_1-1]\cup[r_1+2,n]}(\rk((A_{w})_{S,j})-\rk((A_{w})_{S,j-1}))\lambda_j\\
&\qquad+(\rk((A_{w'})_{S,r_1})-\rk((A_{w'})_{S,r_1-1}))(\lambda_{r_1}+\lambda_{r_1+1}-z_1)\\
&=\sum_{j\in[1,r_1-1]\cup[r_1+2,n]}(\rk((A_{w})_{S^+,j})-\rk((A_{w})_{S^+,j-1}))\lambda_j\\
&\qquad+(\rk((A_{w})_{S^+,r_1+1})-\rk((A_{w})_{S^+,r_1}))(\lambda_{r_1}+\lambda_{r_1+1}-z_1)\\
\end{align*}
by Lemmas \ref{rank_lemma_S+} and \ref{rank_lemma_w'}. 

If $\rk((A_{w'})_{S,r_1})=\rk((A_{w'})_{S,r_1-1})$, or equivalently, if $\rk((A_{w})_{S,r_1})=\rk((A_{w})_{S,r_1-1})$), then using the formula on the second line, the required inequality is exactly the upper bound on $z_S$ in the definition of $\FM(\lambda,w)$.
Indeed, in this case, the right-most columns of $(A_w)_{S,r_1}$ and $(A_w)_{S,r_1+1}$ (which is zero in the latter case) do not increase the rank. 

If instead $\rk((A_{w'})_{S,r_1})-\rk((A_{w'})_{S,r_1-1})=1$, or equivalently, if $\rk((A_{w})_{S^+,r_1+1})-\rk((A_{w})_{S^+,r_1})=1$, then using the formula on the last line, the required inequality follows from the upper bound on $z_{S^+}$ in the definition of $\FM(\lambda,w)$. Therefore, we conclude that $\vz'\in\FM(\lambda',w')$, completing the proof.
\end{proof}

\begin{cor}\label{cor:Zw_dim}
The orbit closure $Z_w=\overline{T\cdot \cL_w}\subset\Fl(n)$ is irreducible of dimension $n-1$.
\end{cor}

\begin{proof}
The irreducibility is immediate from the fact that $Z_w$ is an orbit closure of a point. The moment map image of $Z_w$ is equal to $\FM(\lambda,w)=\GZ(\lambda,w)$, which has dimension $n-1$, hence $Z_w$ has dimension $n-1$.
\end{proof}

\begin{cor}\label{Z_to_Zw}
We have an equality of cycles
\begin{equation*}
[Z]=\sum_{w\in S_{n-1}}[Z_w]
\end{equation*}
on $\Fl(n)$.
\end{cor}

\begin{proof}
The fact that the $Z_w$ all have dimension $n-1$ implies that all of the intermediate subschemes in the degeneration of $Z$ to $Z_w$ have dimension $n-1$, and are components of the corresponding flat limits. Thus, the construction of \S\ref{sec:explicit_degen}, in fact, gives a sequence of toric degenerations of $Z$ into a union of irreducible components, necessarily of multiplicity 1, including all of the $Z_w$ in the end. On the other hand, the moment map images of the $Z_w$, which are equal to the $\GZ(w)$ by Theorem \ref{FM=GZ}, already give a polyhedral subdivision of $\mu(Z)=\Perm(n)$, namely, the HHMP decomposition. Therefore, no other components can appear at the end of this sequence of degenerations, and the claim follows.
\end{proof}

\begin{rem}
The above analysis shows that the data of the components $Z_w$ appearing in the end of our degeneration of $Z$ are essentially encoded by the HHMP decomposition of $\Perm(n)$. The components $Z_w$ are also in bijection with \emph{decreasing binary trees}, which give a recursive way to obtain the HHMP decomposition by slicing by hyperplanes, see \cite[Theorem 6.5]{nt_remixed}. We expect that this slicing procedure more precisely encodes the sequence of intermediate degenerations described in \S\ref{sec:explicit_degen}, but we have not checked this carefully.
\end{rem}

\section{The Anderson-Tymoczko formula}\label{sec:at}

To complete the proof of the Anderson-Tymoczko formula, it suffices to compute $[Z_w]$. We do so by identifying the $Z_w$ as Richardson varieties. Recall that we have fixed a basis $\bC^n=\langle e_1,\ldots,e_n\rangle$. Write in addition $H_i=\langle e_1,\ldots,\widehat{e_i},\ldots,e_n\rangle$. For any subset $S\subset[n]$, write $H_S=\cap_{i\in S}H_i$.

\begin{thm}\label{orbit=richardson}
Let $w\in S_{n-1}$ be a permutation.

Let $F$ be the flag $(0\subset H_{[1,n-1]}\subset H_{[1,n-2]}\subset\cdots\subset H_1\subset\bC^n)$.

Let $F'$ be the transverse flag $(0\subset H_{[2,n]}\subset H_{[3,n]}\subset\cdots\subset H_{n}\subset \bC^n)$. 

Then, the orbit closure $Z_w$ is equal, as a scheme, to the Richardson variety $\Sigma^F_{\iota(w)}\cap\Sigma^{F'}_{\overline{\iota}(w_0w)}$.
\end{thm}

Let us first give representative examples of Theorem \ref{orbit=richardson}.

\begin{ex}
Consider the matrices
\begin{equation*}
A_{\id}=
\begin{bmatrix}
* & * & 0 \\
*& 0 & *\\
* & 0 & 0
\end{bmatrix}
,
A_{(12)}=
\begin{bmatrix}
0 & * & * \\
* & * & 0\\
* & 0 & 0
\end{bmatrix}
.
\end{equation*}
in the case $n=3$, from Example \ref{n=3_matrices}. We see that $Z_{\id}$ is precisely the locus where $L_2$ intersects $H_{[2,3]}=F'_1$ non-trivially, as the vanishing entries in the last column impose no constraint on the flag $\cL$. Accordingly, $\Sigma^F_{\iota(\id)}=\Sigma^F_{\id}$ is all of $\Fl(3)$, and $\Sigma^{F'}_{\overline{\iota}(w_0\cdot\id)}=\Sigma^{F'}_{(23)}$ is the divisor where $\dim(L_2\cap F'_1)\ge 1$. Similarly, $Z_{(12)}$ is the locus where $L_1\subset H_1=F_1$, which is equal to $\Sigma^{F}_{(12)}\cap \Sigma^{F'}_{\id}$.
\end{ex}

\begin{ex}
Let
\begin{equation*}
w=
\begin{bmatrix}
1 & 2 & 3 & 4 & 5 & 6 & 7 \\
3 & 7 & 1 & 2 & 5 & 4 & 6
\end{bmatrix}
\in S_7.
\end{equation*}
Then, from Example \ref{ex:schubert}, a generic point of $\Sigma^F_{\iota(w)}\subset\Fl(8)$ may be represented by an $8\times 8$ matrix
\begin{equation*}
A_F=
\begin{bmatrix}
0 & 0 & * & * & 0 & * & * & * \\
0 & 0 & * & * & 0 & * & * & * \\
* & 0 & * & * & 0 & * & * & * \\
* & 0 & * & * & 0 & * & * & * \\
* & 0 & * & * & * & * & * & * \\
* & 0 & * & * & * & * & * & * \\
* & * & * & * & * & * & * & * \\
* & * & * & * & * & * & * & * 
\end{bmatrix}
\end{equation*}
where the rows correspond to the basis vectors $e_1,\ldots,e_8$ in order. Similarly, we compute that
\begin{equation*}
\overline{\iota}(w_0w)=
\begin{bmatrix}
1 & 2 & 3 & 4 & 5 & 6 & 7 & 8 \\
1 & 6 & 2 & 8 & 7 & 4 & 5 & 3
\end{bmatrix}
\end{equation*}
and that a generic point of $\Sigma^{F'}_{\overline{\iota}(w_0w)}\subset\Fl(8)$ may be represented by an $8\times 8$ matrix
\begin{equation*}
A_{F'}=
\begin{bmatrix}
* & * & * & * & * & * & * & * \\
* & * & * & 0 & * & * & * & * \\
* & * & * & 0 & 0 & * & * & * \\
* & 0 & * & 0 & 0 & * & * & * \\
* & 0 & * & 0 & 0 & * & 0 & * \\
* & 0 & * & 0 & 0 & 0 & 0 & * \\
* & 0 & * & 0 & 0 & 0 & 0 & * \\
* & 0 & * & 0 & 0 & 0 & 0 & * 
\end{bmatrix}
.
\end{equation*}

Theorem \ref{orbit=richardson} asserts that the intersection of these two Schubert subvarieties of $\Fl(8)$ is equal to the torus orbit closure $Z_w$ of the flag $(0\subset L_1\subset \cdots\subset L_7\subset\bC^8)\subset \Fl(8)$ corresponding to the matrix
\begin{equation*}
A_w=
\begin{bmatrix}
0 & 0 & * & * & 0 & 0 & 0 & 0 \\
0 & 0 & * & 0 & * & 0 & 0 & 0 \\
* & * & 0 & 0 & 0 & 0 & 0 & 0 \\
0 & 0 & 0 & 0 & 0 & * & * & 0 \\
0 & 0 & * & 0 & 0 & * & 0 & 0\\
0 & 0 & * & 0 & 0 & 0 & 0 & * \\
* & 0 & * & 0 & 0 & 0 & 0 & 0\\
* & 0 & 0 & 0 & 0 & 0 & 0 & 0
\end{bmatrix}
\end{equation*}
appearing in Example \ref{end_matrix}. The matrix $A_w$ is a specialization of the matrix $A_{F'}$ representing a general point of $\Sigma^{F'}_{\overline{\iota}(w_0w)}\subset\Fl(8)$. However, the same is not true of $A_F$. For example, the second column of $A_w$ does not lie in $H_{[1.6]}=F_2$. However, this is easily rectified by changing basis: for example, some linear combination of the first two columns of $A_w$ lies in $F_2$, so it is still the case that $\dim(L_2\cap F_2)\ge 1$. Similarly, the submatrix of $A_w$ given by the first 4 rows and first 5 columns has rank 3, so we have $\dim(L_5\cap F_4)\ge 2$.
\end{ex}

\begin{proof}[Proof of Theorem \ref{orbit=richardson}]
Both $Z_w$ and $\Sigma^F_{\iota(w)}\cap\Sigma^{F'}_{\overline{\iota}(w_0w)}$ are irreducible and reduced of dimension $n-1$, so it suffices to show that $Z_w\subset\Sigma^F_{\iota(w)}$ and $Z_w\subset\Sigma^{F'}_{\overline{\iota}(w_0w)}$.

We begin by showing that $Z_w\subset\Sigma^{F'}_{\overline{\iota}(w_0w)}$. This amounts to the statement that, for a general point $\cL\in Z_w$, we have
\begin{align*}
\dim(L_i\cap H_{[n+1-j,n]})&\ge\#\left(\{\overline{\iota}(w_0w)(1),\ldots,\overline{\iota}(w_0w)(i)\}\cap\{j+1,\ldots,n\}\right)\\
&=\#\left(\{1,n+1-w(1),\ldots,n+1-w(i-1)\}\cap\{j+1,\ldots,n\}\right)\\
&=\#\left(\{w(1),\ldots,w(i-1)\}\cap\{1,\ldots,n-j\}\right)
\end{align*}
for any $i,j$.
Recall that $L_i$ is the subspace of $\bC^n$ spanned by the first $i$ columns of $A_w$. By construction, the $(k+1)$-st column vector of $A_w$, for $k=1,2,\ldots,i-1$, lies in $H_{[w(k)+1,n]}$, corresponding to the fact that all entries below the symbol $\star_2$ are zero. Therefore, at least $\#\left(\{w(1),\ldots,w(i-1)\}\cap\{1,\ldots,n-j\}\right)$ of the columns $2,\ldots,i$ of $A_w$ are vectors in $H_{[n+1-j,n]}$, establishing the needed inequality.

Now, we show that $Z_w\subset\Sigma^F_{\iota(w)}$. This amounts to the statement that, for a general point $\cL\in Z_w$, we have
\begin{align*}
\dim(L_i\cap H_{[1,j]})&\ge\#\left(\{\iota(w)(1),\ldots,\iota(w)(i)\}\cap\{j+1,\ldots,n\}\right)\\
&=\#\left(\{w(1),\ldots,w(i)\}\cap\{j+1,\ldots,n-1\}\right).
\end{align*}
First, if $w(i)\le j$, then $\dim(L_i\cap H_{[1,j]})\ge \dim(L_{i-1}\cap H_{[1,j]})$ and 
\begin{equation*}
\#\left(\{w(1),\ldots,w(i)\}\cap\{j+1,\ldots,n-1\}\right)=\#\left(\{w(1),\ldots,w(i-1)\}\cap\{j+1,\ldots,n-1\}\right),
\end{equation*}
so we may replace $i$ by $i-1$ and proceed by induction. We therefore assume that $w(i)>j$.

Let $A^{ij}_w$ denote the submatrix of $A_w$ of entries in the first $i$ columns and $j$ rows. Suppose that the symbol $\star_1$ appears in $A^{ij}_w$, in row $\ell\le j$. We claim that the symbol $\star_2$ in the same row as $\star_1$ must also appear in $A^{ij}_w$. Indeed, if this were not the case, then because $\star_2$ appears in column $i+1$ in a row below row $j$, and in particular in no row above row $\ell$, the symbol $\star_1$ would have been placed in row $\ell$ to the right of column $i$, a contradiction.

Let $\alpha=\#\left(\{w(1),\ldots,w(i-1)\}\cap\{j+1,\ldots,n-1\}\right)$ be the number of appearances of the symbol $\star_2$ below row $j$ and between columns $2$ and $i$, inclusive. Then, there are at most $i-1-\alpha$ appearances of the symbol $\star_2$ in $A^{ij}_w$. All appearances of the symbol $\star_1$ in $A^{ij}_w$ must appear in the same rows as the symbols $\star_2$ in $A^{ij}_w$. Thus, for some subset $S\subset\{1,2,\ldots,j\}$ of cardinality at least $j-(i-1-\alpha)$, the corresponding rows of $A^{ij}_w$ are zero.

We therefore have $L_i\subset H_S$. Because $H_{[1,j]}\subset H_S$ has codimension at most $i-1-\alpha$, we also have 
\begin{equation*}
\dim(L_i\cap H_{[1,j]})\ge i-(i-1-\alpha)=\alpha+1,
\end{equation*}
which is exactly the required inequality, because $w(i)>j$ by assumption. This completes the proof.
\end{proof}

\begin{proof}[Proof of Theorem \ref{at_formula}]
Theorem \ref{orbit=richardson} implies that $[Z_w]=\sigma_{\iota(w)}\sigma_{\overline{\iota}(w_0w)}$. The Anderson-Tymoczko formula now follows from Corollary \ref{Z_to_Zw}.
\end{proof}

\section{Grassmannians}\label{sec:grass}

The degeneration of $Z$ to the union of $Z_w$ pushes forward to a toric degeneration of a generic torus orbit closure on any variety of \emph{partial} flags in $\bC^n$ to a union of special orbit closures. We focus on the case of Grassmannians, describing the components that survive under the push-forward and the associated polyhedral subdivision.

Let $\pi:\Fl(n)\to\Gr(r,n)$ be the map remembering only the component $L_r$ of a flag $\cL$. 
\begin{lem}\label{push_zero}
The map $\pi$ has positive-dimensional fibers upon restriction to $Z_w$, and thus sends $[Z_w]$ to zero under push-forward, unless
\begin{equation}\label{perm_ineq}
w(1)>w(2)>\cdots>w(r)=1<w(r+1)<\cdots\cdots<w(n-1).
\end{equation}
\end{lem}

\begin{proof}
We first prove that if $\pi_{*}[Z_w]\neq0$, then we must have $w(1)>w(2)>\cdots>w(r)$. Assume instead that for some $s<r$, we have $w(1)>\cdots>w(s)$ and $w(s)<w(s+1)$. Then, we claim that if the needed conditions
\begin{align*}
\dim(L_i\cap H_{[n+1-j,n]})&\ge(\#\{w(1),\ldots,w(i-1)\}\cap\{1,\ldots,n-j\})\\
\dim(L_i\cap H_{[1,j]})&\ge(\#\{w(1),\ldots,w(i)\}\cap\{j+1,\ldots,n-1\})
\end{align*}
hold for all $i\neq s$ and all $j$, then in fact, they hold for all $i$ and $j$. This implies that if $\cL\in Z_w$, then replacing $L_s$ with any subspace $L'_s$ with $L_{s-1}\subset L'_s\subset L_{s+1}$ yields a flag $\cL'\in Z_w$. If furthermore $s<r$, then this shows that any fiber of $\pi$ upon restriction to $Z_w$ is positive-dimensional, a contradiction.

We now prove the claim. We first have
\begin{align*}
\dim(L_s\cap H_{[1,j]})&\ge\dim(L_{s-1}\cap H_{[1,j]})\\
&\ge\#(\{w(1),\ldots,w(s-1)\}\cap\{j+1,\ldots,n-1\})\\
&=\#(\{w(1),\ldots,w(s)\}\cap\{j+1,\ldots,n-1\})
\end{align*}
unless $w(1)>\cdots>w(s)>j$, in which case we need to prove that $L_s\subset H_{[1,w(s)-1]}$. On the other hand, if $w(s)<w(s+1)$, then we have 
\begin{align*}
\dim(L_{s+1}\cap H_{[1,w(s)-1]})\ge \#(\{w(1),\ldots,w(s+1)\}\cap\{w(s),\ldots,n-1\})=s+1,
\end{align*}
so in fact we have the stronger statement that $L_{s+1}\subset H_{[1,w(s)-1]}$.

Similarly, we have
\begin{align*}
\dim(L_{s}\cap H_{[n+1-j,n]})&\ge\dim(L_{s+1}\cap H_{[n+1-j,n]})-1\\
&\ge(\#\{w(1),\ldots,w(s)\}\cap\{1,\ldots,n-j\})-1\\
&=(\#\{w(1),\ldots,w(s-1)\}\cap\{1,\ldots,n-j\})
\end{align*}
unless $w(1)>\cdots>w(s)>n-j$, in which case the required statement is simply that $\dim(L_{s}\cap H_{[n+1-j,n]})\ge0$. This proves the claim.

Similarly, one proves by downward induction that $w(i)<\cdots<w(n-1)$ for $i\ge r$. 
\end{proof}

If instead \eqref{perm_ineq} holds, then the Schubert variety $\Sigma^{F}_{\iota(w)}\subset\Fl(n)$ pushes forward under $\pi$ to the Schubert variety $\Sigma^F_{\lambda}\subset\Gr(r,n)$, where $\lambda$ is the partition
\begin{equation*}
(w(1)-r,w(2)-(r-1),\ldots,w(r-1)-2,0).
\end{equation*}
On the other hand, the Schubert variety $\Sigma^{F'}_{\overline{\iota}(w)}\subset\Fl(n)$ is the pullback of $\Sigma^{F'}_{\overline{\lambda}}\subset\Gr(r,n)$, where $\overline{\lambda}$ is the complement of $\lambda$ inside the rectangle $(n-r-1)^{r-1}$.

By the projection formula, it follows that the class of a generic torus orbit closure in $\Gr(r,n)$ is given by
\begin{equation*}
\pi_{*}[Z]=\sum_{\lambda\subset(n-r-1)^{r-1}}\sigma_{\lambda}\sigma_{\overline{\lambda}},
\end{equation*}
which was obtained using different methods by Berget-Fink \cite[Theorem 5.1]{bf}.

A typical special orbit closure corresponding to a summand on the right hand side is represented by a $n\times r$ matrix of the form
\begin{equation*}
A_\lambda=
\begin{bmatrix}
0 & 0 & 0 & *\\
0 & 0 & 0 & *\\
0 & 0 & 0 & *\\
0 & 0 & * & *\\
0 & 0 & * & 0\\
0 & * & * & 0\\
0 & * & 0 & 0\\
0 & * & 0 & 0\\
0 & * & 0 & 0\\
* & * & 0 & 0\\
* & 0 & 0 & 0\\
* & 0 & 0 & 0\\
\end{bmatrix}
,
\end{equation*}
obtained by restricting to the first $r$ columns of the matrix $A_w$. The above example corresponds to the permutation $w=(10,6,4,1,2,3,5,7,8,9,11)$ of $n-1=11$, or equivalently the partition $\lambda=(6,3,2)\subset(7)^3$.

The $w(i)$-th row of $A_\lambda$ has non-zero entries in columns $i$ and $i+1$ for $i=1,2,\ldots,r-1$, and each additional row has exactly one non-zero entry in such a way that the non-zero entries in every column are contiguous. In this way, the Schubert conditions corresponding to the cycles $\sigma_{\lambda},\sigma_{\overline{\lambda}}$ are visible, corresponding to the zeroes appearing above and below the path of non-zero entries, respectively.

The degeneration of $\pi(Z)$ into special orbit closures $Z_\lambda\subset\Gr(r,n)$ corresponds to a matroidal decomposition of the simplex $\Delta(r,n)\subset\bR^n$ cut out by the equation $z_1+\cdots+z_n=r$ and the inequalities $0\le z_i\le 1$. Namely, the subpolytopes $\Delta(r,n)_\lambda\subset\Delta(r,n)$ are cut out by the inequalities
\begin{equation*}
z_{[1,w(i)-1]}\le n-i\le z_{[1,w(i)]}
\end{equation*}
for $i=1,2,\ldots,n-1$. This decomposition thus encodes a new proof of the Berget-Fink formula.

\end{document}